\documentclass{amsart}
\usepackage{xspace,amssymb,bbm}
\usepackage[small]{diagrams}
\usepackage{centernot}
\input{command.tex}
\theorems\bbm
\usepackage{hyperref}
\remfalse

\def\mx{\mathrm{max}}
\def\mn{\mathrm{min}}
\def\pp{{\mat{\bP^1}}\xspace}
\def\refl{{\vee\vee}}
\oper{NS}
\def\mm#1{\mu(#1)}
\def\naut#1{\mm{\Aut#1}}
\def\nnaut#1{\n{\Aut#1}}
\def\pf{\cZ} 
\def\pfa{\wtl\cZ}
\def\fam{\bfM}

\def\ndiv{\centernot\mid}
\def\pr{\fS} 
\def\rhom{{\rm R}\lb\Hom}
\def\eq#1{\begin{equation}#1\end{equation}}
\def\eql#1#2{\eq{\label{#2}#1}}
\def\vb{\cM^\circ}
\def\tf{\cM}

\begin{document}
\title[Invariants of moduli spaces]{Invariants of moduli spaces of stable sheaves on ruled surfaces}
\author{Sergey Mozgovoy}%
\email{mozgovoy@maths.tcd.ie}

\begin{abstract}
We compute Betti numbers of the moduli spaces of arbitrary rank stable sheaves on ruled surfaces.
Our result generalizes the formula of G\"ottsche for rank one sheaves and the formula of Yoshioka for rank two sheaves.
It also confirms the conjecture of Manschot for arbitrary rank sheaves on the Hirzebruch surfaces.
\end{abstract}

\maketitle

%

\section{Introduction}
Invariants of the moduli spaces of semistable sheaves on surfaces have been intensively studied in the last several decades.
One of the motivations for their study is the Kobayashi-Hitchin correspondence \cite{lubke_kobayashi} between the above moduli spaces and the moduli spaces of (unframed) instantons on $4$-manifolds. The latter moduli spaces were interpreted by Vafa and Witten \cite{vafa_strong} in terms of the $N=4$ topologically twisted supersymmetric Yang-Mills theory on $4$-manifolds. The $S$-duality conjecture of Vafa and Witten predicts the modular behaviour of the partition function of the above theory. To test this conjecture it is essential to be able to compute the partition function, that is, the generating function of invariants of the moduli spaces of semistable sheaves on a surface. A more recent motivation is related to BPS invariants (or Donaldson-Thomas invariants) of $3$-Calabi-Yau manifolds. The canonical bundle of a surface is a non-compact $3$-Calabi-Yau manifold. In this way one can interpret invariants of the moduli spaces of semistable sheaves on a surface in terms of BPS invariants of the canonical bundle.

Betti numbers of the moduli spaces of rank one sheaves on a surface were computed by G\"ottsche \cite{gottsche_betti}. More precisely, he computed invariants of the Hilbert scheme of points on a surface parameterizing finite subschemes. Any rank one torsion free sheaf with the trivial first Chern class can be uniquely represented as an ideal of a finite subscheme. In this way one can identify the moduli spaces of rank one torsion free sheaves with Hilbert schemes.

Betti numbers of the moduli spaces of rank two sheaves on $\bP^2$ and on ruled surfaces were computed by Yoshioka \cite{yoshioka_betti,yoshioka_bettia}. There are three main ingredients in his approach. First, one computes invariants of the moduli space of semistable torsion free sheaves on a ruled surface with respect to the nef divisor $f$ corresponding to a fiber of the ruled surface. Then one uses wall-crossing formulas to determine invariants of the moduli spaces for arbitrary polarizations of the surface. Finally, to find the invariants of the moduli spaces of rank two semistable sheaves on $\bP^2$, one uses the blow-up formula that relates invariants of the moduli spaces of semistable sheaves on a surface and on its blow-up. The blow-up of $\bP^2$ at one point is the Hirzebruch surface $\Si_1$, where the Hirzebruch surface $\Si_n$, for $n\ge0$, is the ruled surface $\bP(\cO_\pp(n)\oplus\cO_\pp)$ over $\pp$. One can use information about the invariants of the moduli spaces of semistable sheaves on $\Si_1$ to determine invariants of the moduli spaces on $\bP^2$.

Euler numbers of the moduli spaces of rank three semistable sheaves on $\bP^2$ were computed by Weist and Kool \cite{weist_torus,kool_euler} using tori actions. Betti numbers of these moduli spaces were computed by Manschot \cite{manschot_betti} in the case where the first Chern class of semistable sheaves is not divisible by $3$. He used the approach of Yoshioka to reduce the problem to a computation on the Hirzebruch surface $\Si_1$ together with an observation that there are no $f$-semistable rank $3$ sheaves on $\Si_1$ with the required first Chern class. These results together with the blow-up formula were used in \cite{manschot_bps} to determine invariants of the moduli spaces of rank $3$ sheaves on $\bP^2$ and $\Si_1$ for arbitrary first Chern classes. Manschot also formulated a conjecture \cite[Conj.~4.1]{manschot_bps} about the invariants of arbitrary rank $f$-semistable sheaves on the Hirzebruch surfaces.

In this paper we will study invariants of the moduli spaces of arbitrary rank semistable sheaves on a ruled surface. Let $S$ be a surface and $H$ be a nef divisor on $S$. Given a class $\ga=(r,c_1,c_2)\in H^*(S,\bZ)$, let $\cM_H(\ga)$ denote the moduli stack of slope $H$-semistable torsion free sheaves $E$ having rank $r$ and Chern classes $(c_1,c_2)$. Let $\cM_H^\circ(\ga)$ denote its substack of locally free sheaves. Its virtual dimension is 
\[-\hi(E,E)
=2r^2\De(\ga)-r^2\hi(\cO_S),\]
where $\De(\ga)=\De(E)=\frac1r\rbr{c_2+\frac{1-r}{2r}c_1^2}$ is the discriminant of the sheaf $E$, invariant under tensoring with line bundles. We will study the generating function
\[\pfa_H(r,c_1)=\sum_{\ga=(r,c_1,c_2)}\mm{\cM_H(\ga)}t^{r\De(\ga)}\]
and its analogue $\pfa_H^\circ(r,c_1)$ for locally free sheaves,
where $\mm-$ is some motivic measure \cite{kapranov_elliptic}.
Let $S\to C$ be a ruled surface over a curve of genus $g$ and let $f$ be a fiber. A sheaf $E$ over $S$ is slope $f$-semistable if and only if a generic fiber of $E$ along $S\to C$ is semistable, that is, it is a direct sum of line bundles having the same degree.
The main result of the paper is the following theorem

\begin{theorem}
\label{th:main}
Let $r\in\bZ_{>0}$ and $c_1\in H^2(S,\bZ)$. Let \mu be either the Poincar\'e polynomial measure (for $S$ defined over $\bC$) or the point counting measure (for $S$ defined over a finite field).
If $r\ndiv f\cdot c_1$ then $\cM_f(r,c_1,c_2)$ is empty for any $c_2\in\bZ$. Otherwise
\[\pfa_f(r,c_1)
=\mm{\Bun_{C,r}}\prod_{k\ge1}\prod_{i=-r}^{r-1}Z_C(q^{rk+i}t^k),\]
\[\pfa_f^\circ(r,c_1)
=\mm{\Bun_{C,r}}\prod_{k\ge1}\prod_{i=1}^{r-1}\frac{Z_C(q^{rk+i}t^k)}{Z_C(q^{rk-i}t^k)},\]
where $q=\mm{\bA^1}$, $Z_C(t)$ is the motivic zeta function of $C$ associated to $\mm-$, and $\Bun_{C,r}$ is the stack of rank $r$ and degree zero vector bundles over $C$. The motivic measure of $\Bun_{C,r}$ is
\[\mm{\Bun_{C,r}}=\frac{\mm{\Jac C}}{q-1}\prod_{i=1}^{r-1}Z_C(q^i).\]
\end{theorem}


This theorem generalizes the result of G\"ottsche \cite{gottsche_betti} for rank one sheaves and the result of Yoshioka \cite{yoshioka_bettia} for rank two sheaves. It also confirms the conjecture of Manschot \cite{manschot_bps} for the Hirzebruch surfaces. Using this theorem, one can, in principle, apply wall-crossing formulas to compute invariants for the semistable sheaves with respect to any polarization. In particular, one can do this for the Hirzebruch surface $\Si_1$ and then apply the blow-up formula to compute invariants for $\bP^2$. As already the computations for the ranks two and three show \cite{yoshioka_bettia,manschot_bps}, the final result is rather complicated and cumbersome.

%
%

Let me now explain the strategy of the proof of Theorem \ref{th:main}. First, one can reduce the computation for torsion free sheaves to the computation for vector bundles. Similarly to the rank two approach by Yoshioka \cite{yoshioka_bettia}, we will use elementary transformations of vector bundles along fibers to relate invariants of moduli spaces of semistable bundles having different first Chern classes. It turns out that these relations can be written in terms of the Hall algebra of $\bP^1$. More precisely, we will construct a function on the Hall algebra which is uniquely determined by the relations we will impose on it and is such that its integral equals the generating function of the invariants of the moduli spaces we are looking for. The Hall algebra of $\bP^1$ is an extremely well studied object. Its incarnations are the Hall algebra of the Kronecker quiver and, more importantly, the quantum affine algebra $U_q(\what\sl_2)$. Nevertheless, the following result seems to be new. It is crucial for the proof of Theorem \ref{th:main}.

\begin{theorem}
\label{th:hall intro}
Let $\cH$ be the Hall algebra of the category of vector bundles over $\bP^1$ and let $\cH_{r,0}\subset\cH$ be generated by the isomorphism classes of vector bundles having rank $r$ and degree $0$. Then there exists a unique $\bZ[q]$-linear function $\vi:\cH_{r,0}\to\bZ[q]\pser{u,t}$ such that $\vi(\cO^r_\pp)=1$ and such that, for any vector bundles $E,F$ on $\bP^1$ with all summands of $E$ having negative degree, we have
\[\vi([E]\circ[F])=\vi([F]\circ[E])u^{\rk E}t^{-\deg E}.\]
This function satisfies
\[\sum_{\overst{\rk E=r}{\deg E=0}}\vi(E)=\prod_{k\ge1}\prod_{i=1}^{r-1}\frac{1-q^{rk-i}ut^k}{1-q^{rk+i}ut^k}.\]
\end{theorem}

The paper is organized as follows. In Section \ref{sec:prelim} we collect preliminary results on ruled surfaces, Hirzebruch-Riemann-Roch theorem, semistable sheaves, motivic measures, relation between the generating functions for torsion free and locally free sheaves over surfaces, and finally, description of the Hall algebra of the category of vector bundles over \pp. In Section \ref{sec:elem and local} we introduce parabolic bundles over surfaces and their elementary transformations. Then we prove a local version of Theorem \ref{th:main} by reducing it to a computation on the Hall algebra of \pp. This computation is postponed until Section \ref{sec:counts on P1}. In Section \ref{sec:canonical} we show the existence of a canonical filtration of torsion free sheaves on a ruled surface. Then we prove basic structure results about the building blocks of the canonical filtration: slope $f$-semistable sheaves. In Section \ref{sec:global} we prove Theorem \ref{th:main}. In Section \ref{sec:counts on P1} we prove Theorem \ref{th:hall intro}. In Section \ref{sec:wall-cross} we collect the wall-crossing formula and the blow-up formula for the generating functions of invariants of moduli spaces of semistable sheaves over surfaces. We indicate how these formulas can be used to compute the generating functions for $\bP^2$.

I would like to thank Thomas Nevins and Olivier Schiffmann for useful discussions. I would like to thank K{\=o}ta Yoshioka for the help with his paper \cite{yoshioka_chamber}.


\section{Preliminaries}
\label{sec:prelim}
\subsection{Ruled surfaces}
\label{sec:prelim-ruled}
Let $C$ be a smooth projective curve of genus $g$ and let $L$ be a line bundle over $C$ of degree $e\ge0$. Define a ruled surface
\[S=\bP(L\oplus\cO_C)\]
(we use the old-fashioned notation for projective bundles \cite[B.5.5]{fulton_intersection}). Let $p:S\to C$ be the projection map, $L\subset S$ be the canonical embedding, $C_0=S\ms L$ be the divisor at infinity and $f\subset S$ be the divisor of a fiber. Then the N\'eron-Severi group of $S$ is
\[\NS(S)=H^2(S,\bZ)=\bZ C_0\oplus\bZ f.\]
We have \cite[V.2.3-2.9]{hartshorne_algebraic}
\eq{f^2=0,\qquad C_0f=1,\qquad C_0^2=-e.}
If $L'\subset L\oplus\cO_C$ is an embedding of vector bundles and $C'\subset S$ is the corresponding divisor (note that $\deg L'\le e$), then \cite[V.2.9]{hartshorne_algebraic}
\[\deg L'=-C_0\cdot C',\qquad C'\sim C_0+(e-\deg L')f.\]
This implies that any effective divisor in $S$ is of the form $mC_0+nf$ for $m,n\ge0$. 
Note that $(C_0+ef)f=1,\,(C_0+ef)C_0=0$. Therefore the positive cone \[C(S)=\sets{H\in\NS(S)}{HC_0\ge0,\, Hf\ge0}\]
is generated by $C_0+ef$ and $f$.
For any $m,n\in\bR_{\ge0}$, we define
\eql{H_{m,n}=m(C_0+ef)+nf\in C(S).}{eq:Hmn}
The canonical divisor of $S$ is \cite[V.2.11]{hartshorne_algebraic}
\eq{K_S=-2C_0+(2g-2-e)f.}
In particular, $K_S^2=8(1-g)$.

\subsection{Hirzebruch-Riemann-Roch theorem}
Let $S$ be a surface. We will say that a coherent sheaf $E$ on $S$ has class $\ga=(r,c_1,c_2)$ if $\rk E=r$, $c_1(E)=c_1$, and $c_2(E)=c_2$. The second Chern character of $E$ is equal to $\ch_2=\oh c_1^2-c_2$. The Todd class of the tangent sheaf $\cT_S$ of $S$ is given by \cite[\S A.4]{hartshorne_algebraic}
\[\td(S)=(1,-\oh K_S,\hi(\cO_S)),\]
where $K_S$ is the canonical divisor of $S$. The Euler characteristic of the sheaf $E$ equals, by the Hirzebruch-Riemann-Roch theorem, to
\begin{equation}
\hi(E)=\int_S\ch E\td(S)=\ch_2-\oh K_S\cdot c_1+r\hi(\cO_S).
\label{eq:HRR}
\end{equation}
If $F$ is another coherent sheaf having class $\ga'=(r',c'_1,c_2')$ and the second Chern character $\ch'_2=\oh c_1'^2-c'_2$ then
\eq{\hi(E,F)=(r\ch'_2+r'\ch_2- c_1 c'_1)-\oh K_S(r c'_1-r' c_1)+rr'\hi(\cO_S)=:\hi(\ga,\ga').}
In particular,
\eql{\hi(E,E)=2r\ch_2- c_1^2+r^2\hi(\cO_S)}{eq:hi EE}
and
\eql{\ang{E,F}=\hi(E,F)-\hi(F,E)=K_S(r' c_1-r c'_1)=:\ang{\ga,\ga'}.}{eq:ang EF}
Define the discriminant
\eql{\De(\ga)=\De(E)=\frac1r\rbr{c_2+\frac{1-r}{2r}c_1^2}
=\frac{c_1^2}{2r^2}-\frac{\ch_2}{r}
.}{eq:discr}
It is additive under tensoring with a vector bundle and invariant under tensoring with a line bundle (see \cite[\S3.4]{huybrechts_geometry}).
We obtain from \eqref{eq:hi EE} and \eqref{eq:discr} that
\eq{\hi(E,E)=-2r^2\De(E)+r^2\hi(\cO_S).}

\begin{remark}
Let $p:S\to C$ be a ruled surface over a curve of genus $g$. Then
\[\hi(\cO_S)=\hi(Rp_*\cO_S)=\hi(\cO_C)=1-g.\]
\end{remark}

\subsection{Semistable sheaves}
In this section we will define the notions of Gieseker and slope semistability with respect to nef divisors. It should be noted that one usually defines semistability only with respect to ample divisors \cite{huybrechts_geometry}, but it is important for our future considerations to include also the case of nef divisors.

Let $S$ be a projective surface and let $H$ be a nef divisor on $S$. A nef divisor $H$ is characterized by the property that if $E$ is a coherent sheaf on $S$ having dimension one then $H\cdot c_1(E)\ge0$. Given a coherent sheaf $E$ on $S$, define its $H$-slope by
\[\mu_H(E)=\frac{H\cdot c_1(E)}{\rk(E)}\]
and define its reduced Hilbert $H$-polynomial by
\[p_H(E,n)=\frac{\hi(E(nH))}{\rk(E)}.\]
\begin{remark}
By the Hirzebruch-Riemann-Roch Theorem we have
\[\hi(E(nH))=\oh n^2 H^2r+nH(c_1-\frac r2K_S)+\hi(E),\]
where $r=\rk(E)$, $c_1=c_1(E)$. Therefore, for $r>0$, we have
\[p_H(E,n)=\oh n^2 H^2+n\rbr{\mu_H(E)-\frac{HK_S}2}+\frac{\hi(E)}r.\]
\end{remark}

We say that a torsion free sheaf $E$ is slope $H$-semistable if for any proper nonzero torsion free subsheaf $F\subset E$ we have
\[\mu_H(F)\le \mu_H(E).\]
We say that a torsion free sheaf $E$ is Gieseker $H$-semistable if for any proper nonzero torsion free subsheaf $F\subset E$ we have
\[p_H(F,n)\le p_H(E,n),\qquad n\gg0.\]
According to the previous remark the last condition is equivalent to
\begin{enumerate}
	\item either $\mu_H(F)<\mu_H(E)$
	\item or $\mu_H(F)=\mu_H(E)$ and $\frac{\hi(F)}{\rk(F)}\le\frac{\hi(E)}{\rk(E)}$.
\end{enumerate}
In particular, Gieseker semistability implies slope semistability. Both slope semistability and Gieseker semistability can be interpreted in terms of Bridgeland stability conditions on the exact category of torsion free sheaves: for the slope $H$-semistability we consider the stability function
\[Z(E)=-H\cdot c_1(E)+i\rk(E)\]
and for the Gieseker $H$-semistability we consider the stability function
\[Z(E)=-(H\cdot c_1(E)+\eps\hi(E))+i\rk(E)\]
for $0<\eps\ll1$. One can prove the existence and uniqueness of the Harder-Narasimhan filtrations for the slope and Gieseker stability conditions with respect to nef divisors in the same way as with respect to ample divisors \cite{huybrechts_geometry}.

The following simple result is the reason for the wall-crossing formulas on the surface.

\begin{lemma}
\label{lm:ext2}
Let $H$ be a nef divisor such that $H\cdot K_S\le0$.
Let $E,F$ be two slope (or Gieseker) $H$-semistable sheaves such that $\mu_H(E)<\mu_H(F)$. Then $\Ext^2(E,F)=0$.
\end{lemma}
\begin{proof}
By the Serre duality $\Ext^2(E,F)\iso\Hom(F,E\ts\om_S)^*$. But $\mu_H(E\ts\om_S)\le\mu_H(E)<\mu_H(F)$. Therefore $\Hom(F,E\ts\om_S)=0$.
\end{proof}

\subsection{Motivic measures}
Let \bk be a field and $\Sch_\bk$ be the category of schemes of finite type over \bk. Following \cite{kapranov_elliptic}, we define a measure $\mu$ on $\Sch_\bk$ with values in a commutative ring $R$ to be a function which associates with every $X\in\Sch_\bk$ an element $\mu(X)\in\bR$ such that:
\begin{enumerate}
	\item If $U\subset X$ is open then $\mu(X)=\mu(U)+\mu(X\ms U)$.
	\item $\mu(X\xx Y)\iso\mu(X)\mu(Y)$.
\end{enumerate}
We will also assume that $R$ is a \la-ring \cite{getzler_mixed,heinloth_note,mozgovoy_computational} and for any quasi-projective scheme $X\in\Sch_\bk$ we have
\[\mu(\Sym^kX)=\si_k(\mu(X)).\]

\begin{remark}
For $\bk=\bC$ and $R=\bQ[q^\oh]$, define
\[\mu(X)=\sum_{k,i}(-1)^kq^{i/2}\dim\Gr^W_iH_c^k(X,\bC),\]
where $W$ is the weight filtration on $H_c^k(X,\bC)$. If $X$ is smooth and projective then $\mu(X)=\sum_k(-1)^kq^{k/2}\dim H^k(X,\bC)$ is the Poincar\'e polynomial of $X$.
\end{remark}

\begin{remark}
For $\bk=\bF_q$ and $R=\bZ$, define $\mu(X)=\n{X(\bF_q)}$. This function satisfies the first two axioms of the motivic measure. For the last axiom we have to work with the ring of counting sequences \cite{mozgovoy_poincare} instead of the ring $\bZ$.
\end{remark}

We define the zeta function of $X\in\Sch_\bk$ as
\eq{Z_X(t)=\sum_{k\ge0}\si_k(\mu(X))t^k\in R\pser t.}
The ring $\hat R=R\pser t$ has a natural \la-ring structure:
\eq{\si_n(rt^k)=\si_n(r)t^{kn},\qquad \si_n(f+g)=\sum_{k=0}^n\si_k(f)\si_{n-k}(g).}
We define the plethystic exponential $\Exp:\hat R_+\to1+\hat R_+$, where $\hat R_+=t\hat R$, by the formula
\eq{\Exp(f)=\sum_{k\ge0}\si_k(f).}
In particular, $Z_X(t)=\Exp(\mu(X)t)$. The map $\Exp$ has the inverse $\Log:1+\hat R_+\to\hat R_+$ (see \cite{getzler_mixed}). We define a plethystic power map on $\hat R$ by the formula (see \cite{mozgovoy_motivicb} for its basic properties) \eql{f^g=\Exp(g\Log(f)).}{eq:power}

If $C$ is a curve of genus $g$ then $Z_C(t)$ can be written in the form (see \cite{kapranov_elliptic})
\eq{Z_C(t)=\frac{P_C(t)}{(1-t)(1-qt)},}
where $P_C(t)$ is a polynomial of degree $2g$ and $q=\mu(\bA^1)$. Moreover,
\eql{Z_C(t)=(qt^2)^{g-1}Z_C(1/qt).}{eq:Zqt}
The value $P_C(1)$ is equal to $\mu(\Jac C)$. Let $\Bun_{C,r,d}$ denote the moduli stack of vector bundles over $C$ having rank $r$ and degree $d$. Let $\Bun_{C,r}=\Bun_{C,r,0}$. The motive of $\Bun_{C,r,d}$ is independent of $d$ and equals (see \cite[\S6]{behrend_motivica})
\eq{\mu(\Bun_{C,r})=\frac{\mu(\Jac C)}{q-1}q^{(n^2-1)(g-1)}\prod_{i=2}^nZ_C(q^{-i}).}
Applying equation \eqref{eq:Zqt}, we obtain
\eq{\mu(\Bun_{C,r})=\frac{\mu(\Jac C)}{q-1}\prod_{i=1}^{r-1}Z_C(q^i).}

\subsection{Torsion free and locally free sheaves}
\def\maut#1{\mm{\Aut#1}}
\label{sec:tf and lf}
Let \fam be a set of isomorphism classes of sheaves on a surface $S$ having rank $r$ and first Chern class $c_1$. Define
\eql{\pf_\fam(t)=\sum_{E\in\fam}q^{\oh\hi(E,E)}\frac{t^{-\ch_2(E)}}{\maut E},}{eq:pf}
\eql{\pfa_\fam(t)=\sum_{E\in\fam}\frac{t^{r\De(E)}}{\maut E}.}{eq:pfa}
Then
\eql{\pf_\fam(t)=q^{\oh r^2\hi(\cO_S)}t^{-\frac{c_1^2}{2r}}\pfa_\fam(q^{-r}t).}{eq:ch and De}
\rem{$\hi(E,E)=-2r^2\De(E)+r^2\hi(\cO_S)$}
\rem{$r\De(E)=-\ch_2(E)+\frac{c_1^2}{2r}$}

\begin{remark}
Then reason for using $\ch_2(E)$ in \eqref{eq:pf} is that $\ch_2$ is additive with respect to exact sequences. This (as well as the factor $q^{\oh\hi(E,E)}$) will be important in the formulation of the wall-crossing formula (see Prop.~\ref{pr:wall-cross}). The reason for using $r\De(E)$ in \eqref{eq:pfa} is that $r\De(E)$ is invariant under tensoring with line bundles. This will be important in the proof of Corollary \ref{cr:pfa formula}.
\end{remark}

Given a rank $r$ locally free sheaf $E$ over $S$ and $n\ge0$, let $\Quot^n(E)$ denote the scheme of finite quotients of $E$ having length $n$. Yoshioka \cite[Theorem 0.4]{yoshioka_betti} proved that
\eq{\sum_{n\ge0}\mm{\Quot^n(E)}t^n=\prod_{k\ge1}\prod_{i=1}^rZ_S(q^{kr-i}t^k)=:H_r(t).}

Assume that the family \fam consists of locally free sheaves and let $\fam'$ be the set of isomorphism classes of torsion free sheaves $F$ such that $F^\refl\in\fam$.


\begin{lemma}
\label{lm:tf and lf}
We have
\[\pf_{\fam'}(t)=H_r(q^{-r}t)\pf_\fam(t),\qquad 
\pfa_{\fam'}(t)=H_r(t)\wtl\pf_\fam(t).\]
\end{lemma}
\begin{proof}
Let $F\in\fam'$ and $E=F^\refl$, $n=l(E/F)$. Then
\[\ch_2(E)=\ch_2(F)+n.\]
Any automorphism of $F$ induces an automorphism of $E$. Conversely, $\Aut E$ acts on the embeddings $F\emb E$. The stabilizer corresponds to $\Aut F$ and the orbit corresponds to different subobjects of $E$ isomorphic to $F$. Therefore
\[\frac1{\maut F}=\frac{\mm{\sets{U\subset E}{U\iso F}}}{\maut E}.\]
This implies
\[\sum_{F\in\fam'}\frac{t^{-\ch_2(F)}}{\maut F}
=\sum_{E\in\fam}\frac1{\maut E}\sum_{n\ge0}\sum_{\overst{U\subset E}{l(E/U)=n}}t^{-\ch_2(E)+n}
=H_r(t)\sum_{E\in\fam}\frac{t^{-\ch_2(E)}}{\maut E}.
\]
\end{proof}

\subsection{Hall algebra of 
\texorpdfstring{\pp}{P1}}
Given an exact $\bF_q$-linear category \cA with finite $\Hom$ and $\Ext^1$ groups, we define its Hall algebra $\cH(\cA)$ as follows. Its basis is the set of isomorphism classes of the objects in \cA. Multiplication is given by
\eq{[M]\circ[N]=\sum_{[X]}g_{MN}^X[X],}
where
\eq{g_{MN}^X=\n{\sets{U\subset X}{X/U\iso M,\,U\iso N}}.}
It is known, that this product is associative.

Let us describe the Hall algebra $\cH$ of the category of vector bundles on \pp in more detail.
Let $\pr=\cP(\bZ)$ denote the set of maps $\al:\bZ\to\bN$ with finite support
$$\supp\al=\sets{k\in\bZ}{\al_k\ne0}.$$
The basis of $\cH$ is given by the elements $[\cO^\al]$, where
\eq{\cO^\al=\bop_{k\in\bZ}\cO(k)^{\oplus\al_k}.}
Define addition and multiplication in $\pr$ by
\[(\al+\be)_k=\al_k+\be_k,\qquad (\al\be)_k=\sum_{i}\al_i\be_{k-i}.\]
For any $\al\in \pr$, define $\al^*\in \pr$ by $\al^*_k=\al_{-k}$. Given $n\in\bZ$, define $[n]\in \pr$ by $[n]_k=\de_{nk}$.
Then the product $\al[n]\in \pr$ satisfies
$$\al[n]_k=\al_{k-n}.$$
It is clear that
\[\cO^\al\oplus\cO^\be\iso\cO^{\al+\be},\quad  \cO^\al\ts\cO^\be\iso\cO^{\al\be},\quad
(\cO^\al)^*\iso\cO^{\al^*},\quad
\cO^{\al[n]}=\cO^\al\ts\cO(n).\]
For any $\al\in \pr$, define
\[\mn_\al=\min\sets{k\in\bZ}{\al_k\ne0},\qquad \mx_\al=\max\sets{k\in\bZ}{\al_k\ne0}.\]
The product in \cH is described by the following relations \cite[Theorem 10]{baumann_hall}:
\begin{align*}
&[\cO^\al]\circ[\cO^\be]
=[\cO^{\al+\be}],\qquad\text{if }\mx_\al<\mn_\be,\\
&[\cO(n)]^k=[k]^!_q[\cO(n)^{\oplus k}],
\end{align*}
where $[k]^!_q=\prod_{i=1}^k[i]_q$ and $[i]_q=\frac{q^i-1}{q-1}$. Finally, for any $m<n$,
\begin{multline}
[\cO(n)]\circ[\cO(m)]
=q^{n-m+1}[\cO(m)\oplus\cO(n)]\\
+q^{n-m-1}(q^2-1)\sum_{i=1}^{[\frac{n-m}2]}[\cO(m+i)\oplus\cO(n-i)].
\end{multline}

\section{Elementary transformations and the local formula}
\label{sec:elem and local}
\subsection{Elementary transformations}
Let $S$ be a surface, $X\subset S$ be a curve and $I\subset\cO_S$ be an ideal of $X$ in $S$. A parabolic bundle over $(S,X)$ is a triple $(E,M,g)$, where $E$ is a vector bundle over $S$, $M$ is a vector bundle over $X$ and $g:E\to M$ is a surjection. 
We will usually denote a parabolic bundle just as $E\to M$. We define its elementary transformation to be a new parabolic bundle $E'\to M'$ with
$$E'=\ker(E\to M),\qquad M'=\ker(E|_X\to M)$$
and with the canonical surjection $E'\to M'$.
See \cite[\S5.2]{huybrechts_geometry} for the proof that $E'$ is locally free. 
An automorphism of the parabolic bundle $g:E\to M$ is a pair $(f_1,f_2)$ of automorphisms $f_1\in\Aut E$, $f_2\in\Aut M$ such that $gf_1=f_2g$. The groups of automorphisms of $E\to M$ and of $E'\to M'$ are isomorphic.

\begin{lemma}
\label{lm:elem trans}
Let $E''\to M''$ be the elementary transformation of $E'\to M'$.
Then
$$E''\iso I\ts E,\qquad M''\iso I\ts M.$$
\end{lemma}
\begin{proof}
Tensoring exact sequences
$$0\to I\to\cO_S\to\cO_X\to0,\qquad 0\to E'\to E\to M\to0$$
with each other, we obtain
\begin{diagram}
&&&&0&&0\\
&&&&\dTo&&\dTo\\
&&0&\rTo&I\ts E'&\rTo&I\ts E&\rTo&I\ts M&\rTo&0\\
&&&&\dTo&&\dTo&&\dTo_0\\
&&0&\rTo&E'&\rTo&E&\rTo&M&\rTo&0\\
&&&&\dTo&&\dTo&&\dEq\\
0&\rTo&M''&\rTo&E'|_X&\rTo&E|_X&\rTo&M&\rTo&0\\
&&&&\dTo&&\dTo\\
&&&&0&&0
\end{diagram}
Applying the snake lemma to the middle columns, we obtain $M''\iso I\ts M$. Next,
$$E''=\ker(E'\to M')\iso\ker(E\to E|_X)\iso I\ts E.$$
\end{proof}

\subsection{Negative sheaves}
\label{sec:negative}
Assume that $X\iso\bP^1$ and $X\cdot X=0$.
Our main example is a ruled surface $p:S\to C$ with a fiber $X=S_x=p\inv(x)$ over some point $x\in C$.
We have $\deg I|_X=-X\cdot X=0$ and therefore $I|_X\iso\cO_X$.
This implies that in the above lemma we have actually
\begin{equation}
M''\iso I\ts M\iso M.
\label{eq:M''}
\end{equation}

Let $\Coh_XS$ denote the category of coherent sheaves over $S$ having support in $X$. For any such sheaf $F$ we define $r(F)\ge0$ and $d(F)\in\bZ$ by the formulas
\begin{equation}
c_1(F)=r(F)X,\qquad d(F)=\hi(F)-r(F).
\end{equation}

\begin{remark}
If $F$ is a sheaf over $X$ then $r(F)$ is the rank of $F$ over $X$ and $d(F)$ is the degree of $F$ over $X$.
\end{remark}

\begin{remark}
Given $F\in\Coh_XS$ with $r(F)=n$, we have
\[\hi(F)=\ch_2(F)-\oh K_S\cdot c_1(F)+\rk(F)^2\hi(\cO_S)
=-c_2(F)-\oh K_S\cdot nX.\]
By \cite[II.8.20]{hartshorne_algebraic}, we have $\om_X\iso\om_S\ts I^\vee\ts\cO_X=\om_S|_X$. Therefore $K_S\cdot X=-2$. This implies $\hi(F)=-c_2(F)+n$ and $d(F)=-c_2(F)$.
\end{remark}

We will study semistable sheaves in the category $\Coh_XS$ with respect to the stability function
\[Z(F)=-d(F)+ir(F)\]
and the corresponding slope function $\mu(F)=\frac{d(F)}{r(F)}$.
There are obvious semistable sheaves in $\Coh_XS$ coming from the semistable sheaves on $X\iso\pp$.

\begin{proposition}
\label{pr:filt by line bundles}
Any sheaf $F\in\Coh_XS$ has a filtration $0\subset F_0\subset F_1\subset\dots F_r=F$ such that $F_0$ has dimension zero and the quotients $F_i/F_{i-1}$ are line bundles over $X$ with non-increasing degrees.
\end{proposition}
\begin{proof}
We can always find a filtration such that the quotients are sheaves over $X$ and, moreover, are indecomposable (that is, are line bundles or skyscrapers). Our goal is to show that we can reorganize our filtration in such way that all skyscrapers are pushed to the left and the line bundles have non-increasing degrees. To do this we will show that given an extension
\[0\to L\to F\to Q\to0\]
over $S$ with $L$ a line bundle over $X$ and $Q$ a dimension zero sheaf over $X$, the sheaf $F$ is defined over $X$, and given an extension
\[0\to L\to F'\to L'\to0\]
over $S$ with $L,L'$ line bundles over $X$ and $\deg L<\deg L'$, the sheaf $F'$ is again defined over $X$. Then we can exchange a filtration with quotients $(L,Q)$ (resp.\ $(L,L')$) by an appropriate filtration of $F$ (resp.\ $F'$). To prove the first statement, we can assume that $L=\cO_X$. Applying the functor $\Hom_S(Q,-)$ to the exact sequence
\begin{equation}
0\to I\to\cO_S\to\cO_X\to0
\label{eq:ex seq}
\end{equation}
we obtain a long exact sequence
\[\Ext^1(Q,\cO_S)\to\Ext^1(Q,\cO_X)\to\Ext^2(Q,I)\to^0\Ext^2(Q,\cO_S).\]
Note that $\Ext^1(Q,\cO_S)\iso\Ext^1(\cO_S,\om_S\ts Q)^*=0$. Therefore (we use $\om_X\iso\om_S\ts I^\vee\ts\cO_X$ \cite[II.8.20]{hartshorne_algebraic})
\begin{multline*}
\Ext^1(Q,\cO_X)\iso\Ext^2(Q,I)\iso\Hom(\cO_S,\om_S\ts I^\vee\ts Q)^*\\
\iso\Hom(\cO_X,\om_X\ts Q)^*\iso\Ext^1_X(Q,\cO_X),
\end{multline*}
that is, all extensions are defined over $X$. To prove the second statement, we can assume that $L'=\cO_X$ and $\deg L<0$. Applying the functor $\Hom_S(-,L)$ to the exact sequence \eqref{eq:ex seq},
we obtain a long exact sequence
\[\Hom(I,L)\to\Ext^1(\cO_X,L)\to\Ext^1(\cO_S,L)\to^0\Ext^1(I,L).\]
Note that $\Hom(I,L)\iso\Hom(\cO_S,I^\vee\ts L)\iso\Hom(\cO_S,L)=0$ (recall that $I\ts\cO_X\iso\cO_X$). Therefore
\[\Ext^1(\cO_X,L)\iso\Ext^1(\cO_S,L)\iso H^1(X,L)\iso\Ext^1_X(\cO_X,L),\]
that is, all extensions are defined over $X$.
\end{proof}

\begin{corollary}
A semistable object in $\Coh_XS$ either has dimension zero, or has a filtration such that the quotients are line bundles over $X$ having the same degree.
\end{corollary}

Using the Harder-Narasimhan filtrations we can show that there exists a torsion pair $(\cT,\cF)=(\cA^{\ge 0},\cA^{<0})$ on the category $\cA=\Coh_XS$, such that $\cA^{\ge 0}$ is generated by semistable objects having non-negative slope and $\cA^{\ge 0}$ is generated by semistable objects having negative slope.
Note that $\cA^{<0}$ is closed under taking subobjects.
 A sheaf from $\cA^{<0}$ will be called negative. Thus, a sheaf in $\Coh_XS$ is negative if and only if it has a filtration such that all its quotients are line bundles over $X$ having negative degrees. In particular, a sheaf over $X$ is negative if and only if it is a vector bundle and all its summands have negative degrees.

\subsection{Local formula}
As in the previous section, let $X\subset S$ be a curve such that $X\iso\bP^1$ and $X^2=0$. Let $F$ be a rank $r$ locally free sheaf over $S$ such that $F|_X\iso\cO^r_\pp$. 
For any $n\ge0$ and $c_2\in\bZ$, define $A_{F,X}(n,c_2)$ to be the set of locally free sheaves $E\sp F$ (we can consider them as subsheaves $E^\vee\subset F^\vee$) such that
\begin{enumerate}
	\item $E/F\in\Coh_XS$ is negative.
	\item $c_1(E/F)=nX$, $c_2(E/F)=c_2$ (that is, $r(E/F)=n$, $d(E/F)=-c_2$).
\end{enumerate}

\begin{remark}
Our moduli problem should be compared to the one studied by Kapranov \cite{kapranov_elliptic}. He considered a curve $X\subset S$ with $X^2<0$, an $\SL(n)$-bundle $F$ over $S\ms X$, and the moduli space of $\SL(n)$-bundles $E$ over $S$ extending $F$.
\end{remark}

\begin{remark}
Generally, we have $\deg E|_X=X\cdot c_1(E)$. Therefore $X\cdot c_1(F)=\deg\cO_\pp^r=0$ and
\[\deg E|_X=X\cdot c_1(E)=X\cdot c_1(E/F)=0.\]
\end{remark}

\begin{remark}
If $E|_X\iso\cO^r_\pp$ then $E=F$. Indeed, if $E/F\ne0$ then there exists a negative line bundle $L$ over $X$ with a surjection $E/F\to L$. But this implies that there exists a surjection $\cO_\pp^r\iso E|_X\to (E/F)|_X\to L$, which is impossible.
\end{remark}

\begin{theorem}
\label{th:local}
We have
\[\pf_{F,X}
:=\sum_{n\ge0}\sum_{c_2}\sum_{E\in A_{F,X}(n,c_2)}\frac{u^nt^{c_2}}{\nnaut E}
=\frac{1}{\nnaut F}\prod_{k\ge1}\prod_{i=1}^{r-1}\frac{1-q^{rk-i}ut^k}{1-q^{rk+i}ut^k}.\]
\end{theorem}

\begin{proof}
For any vector bundle $E_0$ of rank $r$ and degree zero on $\bP^1$, define
$$A(E_0,n,c_2)=\sets{E\in A_{F,X}(n,c_2)}{E|_X\iso E_0}$$
and the generating function
\[\vi(E_0)
=\sum_{n\ge0}\sum_{c_2}\sum_{E\in A(E_0,n,c_2)}\frac{u^nt^{c_2}}{\nnaut E}
.\]
We extend $\vi$ to the Hall algebra of $\bP^1$ by linearity. 
Note that if $E_0=\cO^r_\pp$ and $E\in A(E_0,n,c_2)$, then $E=F$ and therefore $n=0$, $c_2=0$ and
$$\vi(\cO^r_\pp)=\frac{1}{\nnaut{F}}.$$

Given a pair $(M,M')$ of vector bundles on $\bP^1$, we say that a parabolic bundle $E\to N$ over $(S,X)$ has type $(M,M')$ if
$N\iso M$ and $\ker(E|_X\to N)\iso M'$. We have seen (see Lemma \ref{lm:elem trans} and equation~\eqref{eq:M''}) that elementary transformations of parabolic bundles induce a bijection between parabolic bundles of type $(M,M')$ and parabolic bundles of type $(M',M)$. Assume that $M$ is negative. Let $E\to N$ be a parabolic bundle of type $(M,M')$ with $E\in A_{F,X}(n,c_2)$ and let $E'\to N'$ be its elementary transformation. Then
\[\Hom(F,N)\iso\Hom(F|_X,N)\iso\Hom(\cO^r_\pp,M)=0\]
and therefore $F\subset\ker(E\to N)=E'$. The sheaf $E'/F\subset E/F$ is negative, as a subsheaf of a negative sheaf.
This implies that
\[E'\in A_{F,X}(n-\rk M,c_2+\deg M).\]
Summarizing, for any vector bundles $M,M'$ on $\bP^1$ with negative $M$, we have:
\begin{multline*}
\sum_{[E_0]}\sum_{\overst{U\subset E_0}{E_0/U\iso M,U\iso M'}}
\sum_{E\in A(E_0,n,c_2)}\frac{u^nt^{c_2}}{\nnaut E}\\
=\sum_{[E_0]}\sum_{\overst{U\subset E_0}{E_0/U\iso M',U\iso M}}
\sum_{E\in A(E_0,n-\rk M,c_2+\deg M)}\frac{u^nt^{c_2}}{\nnaut E}.
\end{multline*}
Using generating functions we can write
\[\vi(M\circ M')=u^{\rk M}t^{-\deg M}\vi(M'\circ M),\]
where $\circ$ is the product in the Hall algebra of $\bP^1$.
We will prove in Section \ref{sec:counts on P1} that the last condition implies
$$\pf_{F,X}=\sum_{[E_0]}\vi(E_0)=\frac{1}{\nnaut F}\prod_{k\ge1}\prod_{i=1}^{r-1}\frac{1-q^{rk-i}ut^k}{1-q^{rk+i}ut^k}.$$
\end{proof}
\section{Canonical filtration}
\label{sec:canonical}
Let $p:S\to C$ be a ruled surface as in Section \ref{sec:prelim-ruled}. For each rank $2$ vector bundle over $S$, Brosius \cite{brosius_rank} constructed a canonical short exact sequence with that bundle in the middle, and used such sequences to classify rank $2$ vector bundles over $S$. In this section we will construct canonical filtrations for arbitrary rank torsion free sheaves over $S$. 

Let $E$ be a coherent sheaf over $S$. For any point $x\in C$, we define the fiber $S_x=p\inv(x)$ and we call the restriction $E_x=E|_{S_x}$ the fiber of $E$ along $p$ at the point $x$.
If $E$ is a torsion free sheaf over $S$ then $E^\refl$ is a locally free sheaf and $E^\refl/E$ has dimension zero. This implies that generic fibers of $E$ and $E^\refl$ along $S\to C$ are isomorphic and $c_1(E)=c_1(E^\refl)$. A generic fiber $E_x$ is a vector bundle isomorphic to $\cO_\pp^\al$ for some $\al\in\pr$ independent of $x$.

\begin{theorem}
Let $E$ be a torsion free sheaf over $S$. Then there exists a unique filtration
\[0=F_0\subset F_1\subset\dots\subset F_n=E\]
such that, for any $1\le i\le n$, the sheaf $F_{i}/F_{i-1}$ is torsion free, its generic fiber is isomorphic to $\cO_\pp(k_i)^{r_i}$ for some $k_i\in\bZ$, $r_i\in\bZ_{>0}$, and
\[k_1>k_2>\dots>k_n.\]
The generic fiber of $E$ is isomorphic to $\bop_{i=1}^n\cO_\pp(k_i)^{r_i}$. If $E$ is locally free then every sheaf $F_i$ is locally free.
\end{theorem}
\begin{proof}
Consider the Harder-Narasimhan filtration of $E$ with respect to the slope $f$-stability. The assertion of the theorem follows from the uniqueness of the Harder-Narasimhan filtration and the fact that a rank $r$ torsion free sheaf over $S$ is slope $f$-semistable if and only if its generic fiber is of the form $\cO_\pp(k)^r$ for some $k\in\bZ$. The last statement will be proved in Lemma \ref{lm:f sst}.
\end{proof}

In view of the last theorem, the classification of torsion free sheaves over the ruled surface $S$ is reduced to the classification of $f$-semistable sheaves and extensions between them. We will count $f$-semistable sheaves in the next sections.

\begin{lemma}
Let $E$ be a torsion free sheaf over $S$. Then $p_*E$ is a locally free sheaf.
\end{lemma}
\begin{proof}
Assume that $p_*E$ has a torsion. Then there exists a nonzero morphism $Q\to p_*E$, where $Q$ is a sheaf having dimension zero. This implies that there exists a nonzero morphism $p^*Q\to E$. But $p^*Q$ has dimension at most one, while $E$ is torsion free.
\end{proof}

\begin{lemma}
\label{lm:f sst}
A torsion free sheaf $E$ over $S$ is slope $f$-semistable if and only if a generic fiber of $E$ along $S\to C$ is semistable.
\end{lemma}
\begin{proof}
Assume that a generic fiber of $E$ is semistable. Let $E\to F$ be a torsion free quotient. It induces a surjective morphism of fibers $E_x\to F_x$ for any $x\in C$. By the semistability of a generic fiber of $E$ we obtain (for generic~$x$)
\[\frac{c_1(E)\cdot f}{\rk E}=\frac{\deg E_x}{\rk E_x}\le\frac{\deg F_x}{\rk F_x}=\frac{c_1(F)\cdot f}{\rk F}.\]
Therefore $E$ is slope $f$-semistable.

Let $E$ be slope $f$-semistable and let the generic fiber of $E$ along $S\to C$ be isomorphic to $\cO^\al$ for some $\al\in\pr$.
Assume that $\cO^\al$ is not semistable, that is, $m=\max_\al>\min_\al$. Tensoring $E$ with $\cO_S(-mC_0)$, we can assume that $\max_\al=0$. Let $k=\al_0$.
Then $G=p_*E$ is a locally free sheaf of rank $k$ and there is a natural nonzero morphism $p^*G\to E$. We have seen that $p^*G$ is $f$-semistable. Moreover
\[\frac{c_1(p^*G)\cdot f}{\rk G}=0>\frac{\nn\al}{\n\al}=\frac{c_1(E)\cdot f}{\rk E}.\]
This implies that there are no nonzero morphisms $p^*G\to E$.
\end{proof}

\subsection{Properties of 
\texorpdfstring{$f$}{f}%
-semistable sheaves}
If $E$ is a rank $r$ slope $f$-semistable sheaf then its generic fiber is isomorphic to $\cO_\pp(k)^r$ for some $k\in\bZ$. We have $c_1(E)\cdot f=kr$. The sheaf $E\ts\cO_S(-kC_0)$ is again $f$-semistable and its generic fiber is isomorphic to $\cO_\pp^r$. In this section we will study such sheaves. 

\begin{lemma}
Let $E$ be a locally free sheaf and assume that the set $U\subset C$ of points $x\in C$ such that $E_x\iso\cO^r_\pp$ is open in $C$. Then the natural morphism $p^*p_*E\to E$ is bijective over $p\inv(U)$ and is a monomorphism.
\end{lemma}
\begin{proof}
Let $F=p^*p_*E$ and let $f:F\to E$ be the natural morphism. It induces an isomorphism $p_*F\to p_*E$ and a morphism of fibers $F_x\to E_x$ for any $x\in C$. By the Grauert theorem \cite[III.12.9]{hartshorne_algebraic}, there are isomorhisms
\[
p_*E\ts\bk(x)\iso H^0(S_x,E_x),\qquad
p_*F\ts\bk(x)\iso H^0(S_x,F_x) 
\]
for any $x\in U$. This implies that the maps
\[H^0(S_x,F_x)\to H^0(S_x,E_x)\]
are bijective for $x\in U$ and therefore also the maps of fibers \[F_x\iso\cO^r\to\cO^r\iso E_x\]
are bijective for $x\in U$. This implies that $f$ is surjective over $p\inv(U)$. Therefore $\ker f$ is at most one-dimensional. But $F$ is locally free, hence $\ker f=0$. 
\end{proof}

\begin{lemma}
\label{lm:negative2}
Let $E$ be a locally free sheaf with a generic fiber isomorphic to $\cO^r$. Then $E/p^*p_*E$ is a direct sum of negative sheaves over fibers of $S\to C$.
\end{lemma}
\begin{proof}
Let $F=p^*p_*E$. Then $p_*F=p_*E$ and, by the projection formula, \[R^1p_*F=R^1p_*(\cO_S\ts p^*(p_*E))\iso R^1p_*(\cO_S)\ts p_*E=0.\]
An exact sequence
\[0\to F\to E\to Q\to0\]
induces a long exact sequence
\[0\to p_*F\to p_*E\to p_*Q\to R^1p_*F=0.\]
This implies that $p_*Q=0$ and, by Lemma \ref{lm:negative1}, $Q$ is a direct sum of negative sheaves over fibers of $S\to C$.
\end{proof}

\begin{lemma}
\label{lm:negative1}
Let $F$ be a sheaf over $S$ with a support contained in a fiber $S_x$. Then $F$ is negative if and only if $p_*F=0$.
\end{lemma}
\begin{proof}
If $F$ is negative, then it has a filtration with quotients being negative line bundles over $S_x$. This implies that $p_*F=0$.
Conversely, assume that $p_*F=0$ and $F$ is not negative. 
By Proposition \ref{pr:filt by line bundles}, there exists either a skyscraper $G\subset F$ over $S_x$ or a line bundle $G\subset F$ over $S_x$ with $\deg G\ge0$. In both cases $p_*G\ne0$. This contradicts to $p_*G\subset p_*F=0$.
\end{proof}

\begin{remark}
Let $E$ be a locally free sheaf over $S$ and $G$ be a locally free sheaf over $C$ such that $p^*G\subset E$ and $E/p^*G$ is a direct sum of negative sheaves over fibers of $S\to C$. Then a generic fiber of $E$ is isomorphic to $\cO^r$ and $G\iso p_*E$. We can use this to parametrize locally free sheaves $E$ with a generic fiber isomorphic to $\cO^r$ and with $p_*E\iso G$. 
Applying the duality functor
\[D:D(S)\to D(S),\qquad F\mto\rhom(F,\om_S[2])\]
to the exact sequence
\[0\to p^*G\to E\to Q\to 0\]
we obtain an exact sequence
\[0\to E\dual\ts\om_S\to (p^*G)\dual\ts\om_S\to Q'=\lb\Ext^1(Q,\om_S)\to0.\]
The sheaf $Q'$ is supported on the fibers of $S\to C$ and satisfies the conditions
\[p_*Q'=\lb\Ext^1(R^1p_*Q,\om_C),\qquad R^1p_*Q'=0.\]
Indeed, note that 
$DQ=Q'[1]$, $Rp_*D=DRp_*$, and
\[D(Rp_*Q)=\rhom((R^1p_*Q)[-1],\om_C[1])=\lb\Ext^1(R^1p_*Q,\om_C)[1].\]
Thus, locally free sheaves $E$ as above are parametrized by the quotients $(p^*G)\dual\ts\om_S\to Q'$ (pure of dimension one, to ensure that the kernel is locally free) such that $R^1p_*Q'=0$. One can show, similarly to Lemma \ref{lm:negative1}, that $Q'$ is a direct sum of positive sheaves along fibers (that is, sheaves that are successive extensions of line bundles having positive degrees).
\end{remark}
\section{Counting 
\texorpdfstring{$f$}{f}%
-semistable sheaves}
\label{sec:global}

Let $p:S\to C$ be a ruled surface over a curve $C$ of genus $g$. Given a nef divisor $H$ and a class $\ga=(r,c_1,c_2)$, let $\tf_H(\ga)$ be the moduli stack of slope $H$-semistable torsion free sheaves having rank $r$ and Chern classes $(c_1,c_2)$.
Let $\vb_f(\ga)\subset\tf_f(\ga)$ be the substack of locally free sheaves.
Following Section \ref{sec:tf and lf}, for any $r\ge0$ and $c_1\in\NS(S)$, we define
\eq{\pf_H(r,c_1;t)=\sum_{\overst{c_2\in\bZ}{\ga=(r,c_1,c_2)}}q^{\oh\hi(\ga,\ga)}\mm{\tf_H(\ga)}t^{-\ch_2(\ga)},}
\eq{\pfa_H(r,c_1;t)=\sum_{\overst{c_2\in\bZ}{\ga=(r,c_1,c_2)}}\mm{\tf_H(\ga)}t^{r\De(\ga)}.}
Similarly, we define the series $\pf^\circ_H(r,c_1)$ and $\pfa^\circ_H(r,c_1)$ for locally free sheaves. In this section we will compute the above series for the divisor $H=f$.
We will study first the moduli stack of the form $\vb_f(r,nf,c_2)$ for some $n\in\bZ$.
Note that $c_2=r\De(r,nf,c_2)$ is nonnegative by the Bogomolov inequality \cite[\S3.4]{huybrechts_geometry}, whenever $\vb_f(r,nf,c_2)$ is nonempty.

\begin{theorem}
\label{th:nf}
Let $r>0$ and $n\in\bZ$. Then
\[\pfa^\circ_f(r,nf)=\sum_{c_2}\mm{\vb_f(r,nf,c_2)}t^{c_2}
=\mm{\Bun_{C,r}}\prod_{k\ge1}\prod_{i=1}^{r-1}
\frac{Z_C(q^{rk+i}t^k)}{Z_C(q^{rk-i}t^k)}.\]
\end{theorem}
\begin{proof}
Let $E$ be a rank $r$ slope $f$-semistable locally free sheaf over $S$. Then $G=p_*E$ is a locally free sheaf over $C$, $p^*G\subset E$, and the quotient $E/p^*G$ is a direct sum of negative sheaves along fibers. Let 
\[\cM^\circ_G(n,c_2)\subset\cM^\circ_f(r,(n+\deg G)f,c_2)\]
be the substack of sheaves $E$ such that $p_*E\iso G$. Then $c_1(E/p^*G)=nf$ and $c_2(E/p^*G)=c_2$.

For any $x\in X$, there exists a unique locally free sheaf $p^*G\subset F\subset E$ such that $\supp E/F\subset S_x$ and $F_x\iso\cO_\pp^r$. The sheaf $E/F$ is negative by Lemma \ref{lm:negative2}. Therefore $E\in A_{F,S_x}(n,c_2)$ for some $n,c_2\ge0$ and we can apply Theorem \ref{th:local} to count such sheaves. Varying the point $x\in C$, we obtain
\begin{multline*}
\pf_G:=\sum_{n,c_2}\mm{\cM^\circ_G(n,c_2)}u^nt^{c_2}
=\frac{1}{\naut G}\left(\prod_{k\ge1}\prod_{i=1}^{r-1}\frac{1-q^{rk-i}ut^k}{1-q^{rk+i}ut^k}\right)^{\mm{C}}\\
=\frac{1}{\naut G}\prod_{k\ge1}\prod_{i=1}^{r-1}
\frac{Z_C(q^{rk+i}ut^k)}{Z_C(q^{rk-i}ut^k)},
\end{multline*}
where we used the plethystic power map from equation \eqref{eq:power} and the formula
\[\rbr{\frac1{1-t}}^{\mm C}=\Exp(t)^{\mm C}=\Exp(\mm Ct)=Z_C(t).\]
Summing up over all rank $r$ vector bundles $G$ over $C$, we obtain 
\begin{multline*}
\label{eq:gl1}
\sum_{n,c_2}\mm{\cM_f^\circ(r,nf,c_2)}u^nt^{c_2}
=\sum_{[G]}\pf_G\cdot u^{\deg G}\\
=\sum_{n\in\bZ}\mm{\Bun_{C,r}}u^n\prod_{k\ge1}\prod_{i=1}^{r-1}
\frac{Z_C(q^{rk+i}ut^k)}{Z_C(q^{rk-i}ut^k)}\\
=\sum_{n\in\bZ}\mm{\Bun_{C,r}}u^n\prod_{k\ge1}\prod_{i=1}^{r-1}
\frac{Z_C(q^{rk+i}t^k)}{Z_C(q^{rk-i}t^k)}.
\end{multline*}
Therefore, for any $n\in\bZ$,
\[\sum_{c_2}\mm{\vb_f(r,nf,c_2)}t^{c_2}
=\mm{\Bun_{C,r}}\prod_{k\ge1}\prod_{i=1}^{r-1}
\frac{Z_C(q^{rk+i}t^k)}{Z_C(q^{rk-i}t^k)}.\]
\end{proof}

\begin{corollary}
\label{cr:pfa formula}
Let $r>0$ and $c_1\in \NS(S)$. 
If $r\ndiv f\cdot c_1$ then $\tf_f(r,c_1,c_2)$ is empty for any $c_2\in\bZ$. Otherwise,
\begin{gather*}
\pfa^\circ_f(r,c_1)
=\mm{\Bun_{C,r}}\prod_{k\ge1}\prod_{i=1}^{r-1}
\frac{Z_C(q^{rk+i}t^k)}{Z_C(q^{rk-i}t^k)},\\
\pfa_f(r,c_1)
=\mm{\Bun_{C,r}}\prod_{k\ge1}\prod_{i=-r}^{r-1}Z_C(q^{rk+i}t^k).
\end{gather*}
\end{corollary}
\begin{proof}
Let $c_1=mC_0+nf$ for some $m,n\in\bZ$.
If $\tf_f(r,c_1,c_2)$ is nonempty then, for any $E\in\tf_f(r,c_1,c_2)$, a general fiber $E_x$ is isomorphic to $\cO_\pp(k)^r$ for some $k\in\bZ$. Therefore
\[m=c_1\cdot f=\deg E_x=kr\]
and $r\mid c_1\cdot f$.
The sheaf $E\ts\cO(-kC_0)$ is slope $f$-semistable and has the first Chern class $nf$ and the discriminant $\De(E)$. Therefore
\[\pfa^\circ_f(r,c_1)=\pfa^\circ_f(r,nf)\]
and the first formula of the statement follows. 
To prove the second formula, we will apply Lemma \ref{lm:tf and lf}.
Note that
\[H_r(t)
=\prod_{k\ge1}\prod_{i=1}^rZ_S(q^{rk-i}t^k)
=\prod_{k\ge1}\prod_{i=1}^rZ_C(q^{rk-i}t^k)Z_C(q^{rk-i+1}t^k)
,\]
where we used the fact that $\mm S=\mm C\mm\pp=\mm C(q+1)$ and therefore
\[Z_S(t)=\Exp(\mm St)=Z_C(t)Z_C(qt).\]
Therefore the generating function for torsion free sheaves is
\begin{multline*}
\pfa^\circ_f(r,c_1;t)\cdot H_r(t)
=\mm{\Bun_{C,r}}\prod_{k\ge1}\prod_{i=1}^{r-1}
\frac{Z_C(q^{rk+i}t^k)}{Z_C(q^{rk-i}t^k)}\cdot H_r(t)\\
=\mm{\Bun_{C,r}}\prod_{k\ge1}\prod_{i=-r}^{r-1}Z_C(q^{rk+i}t^k).
\end{multline*}
\end{proof}

\begin{corollary}
Let $r>0$ and $c_1\in \NS(S)$. 
If $r\ndiv f\cdot c_1$ then $\tf_f(r,c_1,c_2)$ is empty for any $c_2\in\bZ$. Otherwise,
\begin{gather*}
\pf^\circ_f(r,c_1)
=q^{\oh r^2(1-g)}t^{-\frac{c_1^2}{2r}}
\mm{\Bun_{C,r}}\prod_{k\ge1}\prod_{i=1}^{r-1}\frac{Z_C(q^{i}t^k)}{Z_C(q^{i}t^k)},\\
\pf_f(r,c_1)
=q^{\oh r^2(1-g)}t^{-\frac{c_1^2}{2r}}\mm{\Bun_{C,r}}\prod_{k\ge1}\prod_{i=-r}^{r-1}Z_C(q^{i}t^k).
\end{gather*}
\end{corollary}

%

\section{Some counts on
\texorpdfstring{$\bP^1$}{P1}}
\label{sec:counts on P1}

Let \cH be the Hall algebra of the category of vector bundles on $\bP^1$.
For any $r>0$ let $\cH_{r,0}\subset\cH$ be the vector space generated by the isomorphism classes of vector bundles having rank $r$ and degree zero.

\begin{theorem}
\label{th:phi}
For any $r>0$, there exists a unique $\bZ[q]$-linear function
$\vi:\cH_{r,0}\to\bZ[q]\pser{u,t}$ such that $\vi(\cO^r_\pp)=1$ and, for any vector bundles $E,F$ on $\bP^1$ such that $\rk E+\rk F=r$, $\deg E+\deg F=0$ and all summands of $E$ have negative degree, we have
\[\vi([E]\circ[F])=\vi([F]\circ[E])u^{\rk E}t^{-\deg E}.\]
This function satisfies
\begin{equation}
\sum_{\overst{\rk E=r}{\deg E=0}}\vi(E)=\prod_{k\ge1}\prod_{i=1}^{r-1}\frac{1-q^{rk-i}ut^k}{1-q^{rk+i}ut^k}.
\label{eq:vi theorem}
\end{equation}
\end{theorem}

Existence of the function \vi follows from the proof of Theorem \ref{th:local}. Uniqueness is straightforward. The rest of this section is devoted to the proof of equation \eqref{eq:vi theorem}. Before starting the proof in full generality let us consider the case $r=2$.
The following result is equivalent to the computation of Yoshioka \cite{yoshioka_bettia}, although the usage of Hall algebras is novel.

\begin{proposition}
For $r=2$, we have
\[\sum_{\overst{\rk E=r}{\deg E=0}}\vi(E)
=\prod_{k\ge1}\frac{1-q^{2k-1}ut^k}{1-q^{2k+1}ut^k}.\]
\end{proposition}
\begin{proof}
For any $n>0$,
$$[\cO(n)]\circ[\cO(-n)]=q^{2n+1}[\cO(-n)\oplus\cO(n)]+q^{2n-1}(q^2-1)\sum_{k=0}^{n-1}[\cO(-k)\oplus\cO(k)].$$
Let $a_n=\vi([\cO(-n)\oplus\cO(n)])$ for $n\ge0$. Then, for any $n>0$,
$$a_n=ut^n\vi([\cO(n)]\circ[\cO(-n)])=ut^n\rbr{q^{2n+1}a_n+(q^{2n+1}-q^{2n-1})\sum_{k=0}^{n-1}a_k}.$$
This implies
\[a_n=\frac{ut^n(q^{2n+1}-q^{2n-1})}{1-ut^nq^{2n+1}}\sum_{k=0}^{n-1}a_k\]
and
$$\sum_{k=0}^na_k
=\rbr{1+\frac{ut^n(q^{2n+1}-q^{2n-1})}{1-ut^nq^{2n+1}}}\cdot \sum_{k=0}^{n-1}a_k
=\frac{1-q^{2n-1}ut^n}{1-q^{2n+1}ut^n}\cdot\sum_{k=0}^{n-1}a_k.$$
Therefore
\[\sum_{k=0}^na_k
=\prod_{k=1}^n\frac{1-q^{2k-1}ut^k}{1-q^{2k+1}ut^k}.\]
\end{proof}

\subsection{Counting quotients on 
\texorpdfstring{\pp}{P1}}
\label{sec:quotients on pp}
Let $E$ be a coherent sheaf on \pp and let $\Quot(E)$ be the Grothendieck Quot-scheme of $E$. Given another coherent sheaf $F$ on \pp, let $\Quot(E,F)$ be the subscheme of $\Quot(E)$ corresponding to the epimorphisms $E\to F$. It is a natural problem to compute the motive of $\Quot(E,F)$, as both $E$ and $F$ are discretely parametrized (apart from the torsion parts). In this section we will do this in the case when $E$ is a vector bundle and $F$ is a line bundle. So, we assume that $E=\cO^{\al}$ for some $\al\in\pr$, and $F=\cO(n)$ for some $n\in\bZ$. Denote by $g(\al,n)$ the motive of $\Quot(E,F)$, that is, the number of epimorphisms $E\to F$ up to the action of $\Aut F$.

\begin{proposition}
For any $\al\in \pr$ and $n\in\bZ$, we have
\[g(\al,n)
=\frac1{q-1}q^{\sum_{k\le n}(n-k+1)\al_k}\left(1-(q+1)q^{-\sum_{k\le n}\al_k}+q^{1+\al_n-2\sum_{k\le n}\al_k}\right).\]
\end{proposition}
\begin{proof}
It is enough to prove the formula for $g(\al)=g(\al,0)$. Let
\[h^0(\al,n)=\dim\Hom(\cO^\al,\cO(n))=\sum_{k\le n}(n-k+1)\al_k.\]
Any nonzero morphism $\cO^\al\to\cO$ can be uniquely written (up to the action of \Gm) as a composition of a surjection $\cO^\al\to\cO(n)$ and an embedding $\cO(n)\to\cO$ for some $n\le0$. Therefore
\[q^{h^0(\al,0)}-1=\sum_{n\le0}g(\al,n)(q^{h^0(\cO(n),\cO)}-1)\]
or, equivalently,
\[q^{h^0(\al,0)}-1=\sum_{n\ge0}g(\al[n])(q^{n+1}-1).\]
Applying this to $\al[1]$, we obtain
\[q^{h^0(\al[1],0)}-1
=\sum_{n\ge0}g(\al[n+1])(q^{n+1}-1)
=\sum_{n\ge0}g(\al[n])(q^{n}-1).\]
The last two formulas imply
\[\sum_{n\ge0}g(\al[n])(q-1)
=q^{h^0(\al,0)}-q^{h^0(\al[1],0)+1}+q-1
.\]
Applying this to $\al[1]$, we obtain
\[\sum_{n\ge0}g(\al[n+1])(q-1)
=q^{h^0(\al[1],0)}-q^{h^0(\al[2],0)+1}+q-1
.\]
Finally, subtracting the last formula from the previous one, we get
\[g(\al)=\frac{1}{q-1}\rbr{q^{h^0(\al,0)}-(q+1)q^{h^0(\al[1],0)}
+q^{h^0(\al[2],0)+1}}.\]
We note that
\[h^0(\al[1],n)=\sum_{k\le n-1}(n-1-k+1)\al_k
=\sum_{k\le n}(n-k+1)\al_k-\sum_{k\le n}\al_k,\]
\[h^0(\al[2],n)=\sum_{k\le n-2}(n-2-k+1)\al_k
=\sum_{k\le n}(n-k+1)\al_k-2\sum_{k\le n}\al_k+\al_n.\]
\end{proof}

\begin{corollary}
\label{cr:surj}
Assume that $\n\al=r$, $\nn\al=d$, and $n>\mx_\al$. Then
\[g(\al,n)=q^{nr-r-d+1}\frac{(q^r-1)(q^{r-1}-1)}{q-1}.\]
\end{corollary}
\begin{proof}
By the previous proposition we have
\begin{multline*}
g(\al,n)
=\frac1{q-1}q^{\sum_{k}(n-k+1)\al_k}
\rbr{1-(q+1)q^{-\sum_{k}\al_k}+q^{1-2\sum_{k}\al_k}}\\
=\frac{q^{(n+1)r-d}}{q-1}\rbr{1-(q+1)q^{-r}+q^{1-2r}}
=\frac{q^{(n-1)r-d}}{q-1}(q^r-1)(q^r-q).
\end{multline*}
\end{proof}

\subsection{Skew derivations}
Given a line bundle $L$ over $\bP^1$ define the skew derivation $\de_L$ on the Hall algebra by
\[\de_L(F)=q^{-\hi(L,F)}[F]\circ[L]-[L]\circ[F].\]
In this section we will compute skew derivations of some elements in the Hall algebra.

\begin{proposition}
For any $r>0$ and $d,n\in\bZ$, let
\[B_{r,d,n}=\sum_{\over{\n\al=r,\nn\al=d}{\mn_\al>n}}[\cO^{\al}].\]
Then
\[\de_{\cO(n)}(B_{r,d,n})
=f_r B_{r+1,d+n,n},\qquad f_r=q^{-2r}\frac{(q^r-1)(q^{r+1}-1)}{q-1}.\]
%
\end{proposition}
\begin{proof}
It is enough to prove the statement for $n=0$ and $B_{r,d}=B_{r,d,0}$. Let
\[\pr_{r,d}^+=\sets{\al\in \pr}{\n\al=r,\nn\al=d,\mn_\al>0}.\]
Multiplication rules in the Hall algebra imply that, for any $\al\in \pr_{r,d}^+$,
\[[\cO^\al]\circ[\cO]=q^{\hi(\cO,\cO^\al)}[\cO]\circ[\cO^\al]+\sum_{\be\in \pr_{r+1,d}^+}c_\be[\cO^\be]\]
for some coefficients $c_\be$ (note that $\hi(\cO,\cO^\al)=d+r$). Similarly,
\[B_{r,d}\circ[\cO]-q^{d+r}[\cO]\circ B_{r,d}=\sum_{\be\in \pr_{r+1,d}^+}c_\be[\cO^\be]\]
for some coefficients $c_\be$. For any $\be\in \pr_{r+1,d}^+$, the coefficient of $\cO^\be$ in the product $B_{r,d}\circ[\cO]$ equals the number of embeddings $\cO\subset \cO^\be$ such that the quotient is locally free. This number is equal to $g(\be^*,0)$ introduced in Section \ref{sec:quotients on pp}. By Corollary \ref{cr:surj} we have
\[g(\be^*,0)=q^{-r+d}\frac{(q^r-1)(q^{r+1}-1)}{q-1}.\]
This implies
\begin{multline*}
q^{-d-r}B_{r,d}\circ[\cO]-[\cO]\circ B_{r,d}=
q^{-d-r}\sum_{\be\in \pr_{r+1,d}^+}g(\be^*,0)[\cO^\be]\\
=q^{-2r}\frac{(q^r-1)(q^{r+1}-1)}{q-1} B_{r+1,d}.
\end{multline*}
\end{proof}

\subsection{Proof of the theorem}
We assume that $r>0$ is fixed.

\begin{proposition}
Given $n\ge0$, let
\[\pr^0_n=\sets{\al\in \pr}{\n\al=r,\nn\al=0,\mn_\al\ge-n}.\]
Then
\begin{equation}
\sum_{\al\in \pr^0_{n}}\vi(\cO^\al)
=\prod_{k=1}^n\prod_{i=1}^{r-1}\frac{1-q^{rk-i}ut^k}{1-q^{rk+i}ut^k}.
\label{eq:phi n}
\end{equation}
\end{proposition}
\begin{proof}
For any $0\le k\le r-1$, define
\[\pr^0_{n,k}=\sets{\al\in \pr^0_n}{\al_{-n}=k}\]
and
\[A_{n,k}
=\sum_{\al\in \pr^0_{n,k}}[\cO^\al]
=[\cO(-n)^{\oplus k}]\circ B_{r-k,kn,-n}.\]
We know that $[\cO(n)^{\oplus k}]=\frac{1}{[k]^!_q}[\cO(n)]^k$.
Define
\[A'_{n,k}=\frac{1}{[k]^!_q}[\cO(-n)]^{k-1}\circ B_{r-k,kn,-n}
=\frac{1}{[k]_q}[\cO(-n)^{\oplus k-1}]\circ B_{r-k,kn,-n},\]
so that $A_{n,k}=[\cO(-n)]\circ A'_{n,k}$.
If $\n\al=r-k$, $\nn\al=nk$ then
\[\hi(\cO(-n),\cO^\al)=nk+n(r-k)+r-k=nr+r-k.\]
Therefore
\begin{multline*}
q^{-nr-r+k}A'_{n,k}\circ[\cO(-n)]-[\cO(-n)]\circ[A'_{n,k}]\\
=\frac{f_{r-k}}{[k]_q}[\cO(-n)^{\oplus k-1}]\circ B_{r-k+1,kn-n,-n}
=\frac{f_{r-k}}{[k]_q}A_{n,k-1}.
\end{multline*}
Applying the function \vi, we obtain
\[(q^{-nr-r+k}u\inv t^{-n}-1)\vi(A_{n,k})=\frac{f_{r-k}}{[k]_q}\vi(A_{n,k-1}).\]
Let $z=q^{nr}ut^n$. Then
\[\vi(A_{n,k})
=\frac{f_{r-k}}{(q^{k-r}z\inv-1)[k]_q}\vi(A_{n,k-1})
=\frac{zq^{k-r}(q^{r-k}-1)(q^{r-k+1}-1)}{(1-zq^{r-k})(q^k-1)}\vi(A_{n,k-1}).\]
\rem{$r-k-2(r-k)=k-r$}
Therefore
\[\sum_{\al\in \pr_n^0}\vi(\cO^\al)=\sum_{k=0}^{r-1}\vi(A_{n,k})
=\vi(A_{n,0})\sum_{m=0}^{r-1}\prod_{k=1}^m
\frac{zq^{k-r}(q^{r-k}-1)(q^{r-k+1}-1)}{(1-zq^{r-k})(q^k-1)}
\]
while $\vi(A_{n,0})=\sum_{\al\in \pr_{n-1}^0}\vi(\cO^\al)$. Our proposition will be proved by induction if we will show that
\[\sum_{m=0}^{r-1}\prod_{k=1}^m
\frac{zq^{k-r}(q^{r-k}-1)(q^{r-k+1}-1)}{(1-zq^{r-k})(q^k-1)}
=\frac{(1-zq\inv)\dots(1-zq^{1-r})}{(1-zq)\dots(1-zq^{r-1})}\]
or, equivalently,
\[\sum_{m=0}^{r-1}\prod_{k=1}^m\frac{(1-q^{r-k})(1-q^{r+1-k})}{(1-zq^{r-k})(1-q^{-k})}(zq^{-r})^m=\frac{(1-zq\inv)\dots(1-zq^{1-r})}{(1-zq)\dots(1-zq^{r-1})}.\]
Inverting $q$, we can write this equation in the form
\[\sum_{m=0}^{r-1}\frac{(q^{-r+1};q)_m(q^{-r};q)_m}{(zq^{-r+1};q)_m(q;q)_m}(zq^{r})^m=\frac{(1-zq)\dots(1-zq^{r-1})}{(1-zq\inv)\dots(1-zq^{1-r})},\]
where $(z;q)_m=\prod_{i=0}^{m-1}(1-zq^i)$.
The last equation follows from the Heine-Gauss summation formula
\cite[1.5.1]{gasper_basic}

\def\qhyper#1#2#3#4#5{{}_#1\phi_#2\left(\left.\begin{matrix}#3\\ #4\end{matrix}\,\right|#5\right)}

$${}_2\phi_1\left(\left.\begin{matrix}a,b\\ c\end{matrix}\, \right| q;\frac c{ab}\right)=\frac{(c/a;q)_\infty(c/b;q)_\infty}{(c;q)_\infty(c/ab;q)_\infty},$$
where $(z;q)_\infty=\prod_{i=0}^{\infty}(1-zq^i)$ and the $q$-hypergeometric series ${}_2\phi_1$ is defined by 
\cite[1.2.14]{gasper_basic}
\[\qhyper 21{a,b}c{q;z}=\sum_{m\ge0}\frac{(a;q)_m(b;q)_m}{(c;q)_m(q;q)_m}z^m.\]
\end{proof}

Taking the limit in the equation \eqref{eq:phi n} as $n\to\infty$ we obtain
\[\sum_{\n\al=r,\nn\al=0}\vi(\cO^\al)=\prod_{k\ge1}\prod_{i=1}^{r-1}\frac{1-q^{rk-i}ut^k}{1-q^{rk+i}ut^k}.\]
This proves Theorem \ref{th:phi}.
\section{Wall-crossing and blow-up formulas}
\label{sec:wall-cross}
\subsection{Wall-crossing formula}
Let $S$ be a surface and let $\Ga=\bZ\xx\NS(S)\xx\bZ$ and
\eq{\Ga_+=(\bZ_{>0}\xx\NS(S)\xx\bZ)\cup(\set0\xx(\NS^+(S)\ms0)\xx\bZ)\cup(\set0\xx\set0\xx\bZ_{\ge0}),}
where $\NS^+(S)\subset\NS(X)$ is an effective monoid, generated by the classes of irreducible curves in $S$.
For any elements $\ga=(r,c,d)$, $\ga'=(r',c',d')$ in \Ga, we define, following \eqref{eq:ang EF},
\eq{\ang{\ga,\ga'}=K_S(r'c-rc').}
Let $\bA_S$ be a suitable completion of $\bQ(q^\oh)[\Ga_+]$ (see e.g., \cite[\S3]{bridgeland_hall}) with a basis $x^\ga=v^{r}u^{c}t^{d}$ for $\ga=(r,c,d)\in\Ga_+$, and multiplication
\eq{x^\ga\circ x^{\ga'}=q^{\oh\ang{\ga,\ga'}}x^{\ga+\ga'}.}
For any nef divisor $H$, integer $r\ge0$, and element $c\in\NS(X)$, we defined
\eq{\pf_H(r,c)
=\sum_{\ga=(r,c,d)}q^{\oh\hi(\ga,\ga)}\mm{\cM_H(\ga)}t^{-\ch_2(\ga)}
.}
Given another nef divisor $H'$, we define $H_{\pm}=H\pm\eps H'$ for $0<\eps\ll1$. Define $\mu_H(r,c)=\frac{H\cdot c}{r}$.

\begin{proposition}[Wall-crossing formula]
\label{pr:wall-cross}
Assume that $H\cdot K_S<0$. Then
\[\pf_H(r,c)
=\sum_{\overst{(r,c)=\sum(r_i,c_i)}{\overst{\mu_H(r,c)=\mu_H(r_i,c_i)}{\mu_{H'}(r_1,c_1)>\dots>\mu_{H'}(r_k,c_k)}}}
q^{\oh\sum_{i<j}K_S(r_jc_i-r_ic_j)}\prod_{i=1}^k\pf_{H_+}(r_i,c_i).\]
If $H_-$ is nef then also
\[\pf_H(r,c)
=\sum_{\overst{(r,c)=\sum(r_i,c_i)}{\overst{\mu_H(r,c)=\mu_H(r_i,c_i)}{\mu_{H'}(r_1,c_1)<\dots<\mu_{H'}(r_k,c_k)}}}
q^{\oh\sum_{i<j}K_S(r_jc_i-r_ic_j)}\prod_{i=1}^k\pf_{H_-}(r_i,c_i).\]
\end{proposition}
\begin{proof}
There exists a canonical map $I$ (called integration map) from the (opposite) Hall algebra of the category of coherent sheaves over $S$ to the quantum affine plane $\bA_S$ (see e.g., \cite{mozgovoy_wall-crossing}). This map does not preserve products in general, but if $\Ext^2(E,F)=0$, then $I([F]\circ[E])=I([F])\circ I([E])$. For any $H$-semistable sheaf, we consider its Harder-Narasimhan filtration with respect to $H_+$. In this way we obtain a relation in the Hall algebra between $H$-semistable sheaves and $H_+$-semistable sheaves. By Lemma \ref{lm:ext2} the second extension groups between factors of Harder-Narasimhan filtrations are zero. Therefore the integration map transforms the above relation in the Hall algebra into a relation in the quantum affine plane $\bA_S$. This gives the first formula of the proposition. The proof of the second formula is the same.
\end{proof}

\begin{remark}
Assume that $S=\Si_n$ is a Hirzebruch surface. Then $K_S=-2C_0-(2+n)f$ and for any nef divisor $H\ne0$ we have $H\cdot K_S<0$. This means that the wall-crossing formula is always satisfied. A similar wall-crossing formula for the Hirzebruch surfaces can be found in \cite{manschot_bps}.
\end{remark}

\subsection{Blow-up formula}
Let $S$ be a smooth projective surface, $\pi:\what S\to S$ be the blow-up at a point and $C_0$ be the exceptional divisor of \pi.
As in Section \ref{sec:tf and lf}, let $\fam$ be a set of isomorphism classes of locally free sheaves on $S$ having rank $r$ and the first Chern class $c_1$. For any $m\in\bZ$, define
\begin{enumerate}
\item $\fam'$ to be the set of isomorphism classes of torsion free sheaves $E$ over $S$ such that $E^\refl\in\fam$.
\item $\what\fam$ to be the set of isomorphism classes of locally free sheaves $E$ over $S$ such that $\pi_*E\in\fam'$ and $c_1(E)=\pi^*c_1-mC_0$ (that is, $c_1(E)\cdot C_0=m$).
\item $\what\fam'$ to be the set of isomorphism classes of torsion free sheaves $E$ over $\what S$ such that $\pi_*E\in\fam'$ and $c_1(E)=\pi^*c_1-mC_0$ (equivalently, if $E^\refl\in\what\fam$).
\end{enumerate}

\rem{If $E$ is torsion free over $\what S$ then $\pi_*E$ is torsion free over $S$. Indeed, if $Q\subset\pi_*E$ is torsion then $\pi^*Q\to E$ is nonzero, but the dimension of $\pi^*Q$ is at most one}

\rem{If $E$ is locally free over $\what S$ then $(\pi_*E)^\refl=F$ if and only if $E\in\pi^*F$ and $\pi^*F/E$ is supported on $C_0$.}

The following result was proved by Yoshioka \cite[Prop.~3.4]{yoshioka_chamber} (see also \cite[\S3]{gottsche_theta}).

\begin{proposition}
We have
\[
\pfa_{\what\fam'}(t)
=\prod_{k\ge1}\frac{1}{(1-q^{rk}t^k)^r}
\sum_{\overst{\sum_{i=1}^r a_i=0}{a_i\in\bZ+\frac mr}}
q^{\sum_{i<j}\binom{a_j-a_i}{2}}t^{-\sum_{i<j}a_ia_j}
\pfa_{\fam'}(t).
\]
\end{proposition}

\begin{corollary}
We have
\eql{
\frac{\pf_{\what\fam'}(t)}{\pf_{\fam'}(t)}
=\prod_{k\ge1}\frac1{(1-t^k)^r}
\sum_{\overst{a\in\bZ^r}{\sum_{i=1}^r a_i=-m}}q^{(\rho,a)}t^{\oh(a,a)},
}{blow tf}
where $\rho=\oh\sum_{i<j}(e_i-e_j)\in\bR^r$ and $(a,b)=\sum_{i=1}^r a_ib_i$ for $a,b\in\bR^n$.
\end{corollary}
\begin{proof}
Let $\sum_{i=1}^ra_i=0$ and $a_i\in\bZ+\frac mr$. Define $a'_i=a_i-\frac mr$. Then $\sum_{i=1}^ra'_i=-m$ and
\[(a',a')=(a,a)+\frac{m^2}{r}.\]
We have
\[-\sum_{i<j}a_ia_j=\oh\sum a_i^2=\oh(a,a),\]
\begin{multline*}
\sum_{i<j}\binom{a_j-a_i}{2}=\oh\sum_{i<j}(a_j-a_i)^2
-\oh\sum_{i<j}(a_j-a_i)\\
=\frac{r-1}{2}(a,a)-\sum_{i<j}a_ia_j+(\rho,a)
=\frac r2(a,a)+(\rho,a).
\end{multline*}
If $c_1'=\pi^*c_1-mC_0$ then
\[-\frac{c_1'^2}{2r}=-\frac{c_1^2}{2r}+\frac{m^2}{2r}.\]
Applying equation \eqref{eq:ch and De}, we obtain
\begin{multline*}
\frac{\pf_{\what\fam'}(t)}{\pf_{\fam'}(t)}
=\frac{t^{-\frac{c_1'^2}{2r}}\pfa_{\what\fam'}(q^{-r}t)}{t^{-\frac{c_1^2}{2r}}\pfa_{\fam'}(q^{-r}t)}
=t^{\frac{m^2}{2r}}\prod_{k\ge1}\frac{1}{(1-t^k)^r}
\sum_{\overst{\sum a_i=0}{a_i\in\bZ+\frac mr}}q^{\frac r2(a,a)+(\rho,a)}(q^{-r}t)^{\oh(a,a)}\\
=\prod_{k\ge1}\frac{1}{(1-t^k)^r}
\sum_{\sum a'_i=-m}q^{(\rho,a')}t^{\oh(a',a')}.
\end{multline*}
\end{proof}

\begin{corollary}
We have
\[
\frac{\pf_{\what\fam}(t)}{\pf_{\fam}(t)}
=\prod_{k\ge1}\prod_{i=1}^{r-1}\frac{1-q^{-i}t^k}{1-t^k}
\sum_{\sum_{i=1}^r a_i=-m}q^{(\rho,a)}t^{\oh(a,a)}.
\]
\end{corollary}
\begin{proof}
According to Lemma \ref{lm:tf and lf}, we have $\pf_{\what\fam'}(t)=\pf_{\what\fam}(t)H_{\what S,r}(q^{-r}t)$, where
\[H_{\what S,r}(q^{-r}t)
=\prod_{k\ge1}\prod_{i=1}^rZ_{\what S}(q^{-i}t^k)\]
and $\pf_{\fam'}(t)=\pf_{\fam}(t)H_{S,r}(q^{-r}t)$, where
\[H_{S,r}(q^{-r}t)
=\prod_{k\ge1}\prod_{i=1}^rZ_{S}(q^{-i}t^k).\]
Note that $\mm{\what S}=\mm S+q$. Therefore
\[Z_{\what S}(t)=Z_S(t)Z_{\bA^1}(t)=Z_S(t)\frac{1}{1-qt}.\]
This implies
\begin{multline*}
\frac{\pf_{\what\fam}(t)}{\pf_{\fam}(t)}
=\frac{\pf_{\what\fam'}(t)}{\pf_{\fam'}(t)}
\frac{H_{S,r}(q^{-r}t)}{H_{\what S,r}(q^{-r}t)}
=\frac{\pf_{\what\fam'}(t)}{\pf_{\fam'}(t)}
\prod_{k\ge1}\prod_{i=1}^r(1-q^{-i+1}t^k)\\
=\prod_{k\ge1}\rbr{\frac1{(1-t^k)^r}\prod_{i=1}^r(1-q^{-i+1}t^k)}
\sum_{\sum a_i=-m}q^{(\rho,a)}t^{\oh(a,a)}\\
=\prod_{k\ge1}\prod_{i=1}^{r-1}\frac{1-q^{-i}t^k}{1-t^k}
\sum_{\sum a_i=-m}q^{(\rho,a)}t^{\oh(a,a)}.
\end{multline*}
\end{proof}

\begin{remark}
A generalization of the above result (for $m=0$) to principal bundles with respect to arbitrary reductive groups was proved by Kapranov \cite[Theorem 7.4.6]{kapranov_elliptic}. Note, however, that his formula is slightly different from the above result.
\end{remark}

Let $H$ be a nef divisor on $S$. Then $\what H=\pi^*H$ is a nef divisor on $\what S$ and we can consider slope $\what H$-semistable sheaves on $\what S$.

\begin{corollary}
For any $r\ge0$, $c_1\in\NS(S)$, and $m\in\bZ$, we have
\[
\frac{\pf_{\pi^*H}(r,\pi^*c_1+mC_0)}{\pf_H(r,c_1)}
=\prod_{k\ge1}\frac1{(1-t^k)^r}
\sum_{\sum_{i=1}^r a_i=m}q^{(\rho,a)}t^{\oh(a,a)},
\]
\[
\frac{\pf^\circ_{\pi^*H}(r,\pi^*c_1+mC_0)}{\pf^\circ_H(r,c_1)}
=\prod_{k\ge1}\prod_{i=1}^{r-1}\frac{1-q^{-i}t^k}{1-t^k}
\sum_{\sum_{i=1}^r a_i=m}q^{(\rho,a)}t^{\oh(a,a)}.
\]
\end{corollary}

\begin{remark}
The blow-up $\pi:\Si_1\to\bP^2$ of $\bP^2$ at one point is a Hirzebruch surface. Note that $\Si_1=\bP(\cO_\pp(1)\oplus\cO_\pp)$ is a ruled surface over $\pp$. Let $C_0,\,f$ be divisors on $\Si_1$ introduced in Section \ref{sec:prelim-ruled}.
One can show that, for a line divisor $H$ on $\bP^2$,  $\pi^*H=C_0+f=H_{1,0}$ (see \eqref{eq:Hmn}). In order to determine $\pf_{H}(r,c_1)$ using the previous corollary we have to compute $\pf_{H_{1,0}}(r,\pi^*c_1+mC_0)$ for some $m\in\bZ$. We know how to compute these invariants with respect to the nef divisor $f=H_{0,1}$ (see Theorem \ref{th:main}). Using this result together with the wall-crossing formula and the computation of Zagier \cite{zagier_elementary}, one can compute invariants for the divisor $H_{\eps,1}$ for $0<\eps\ll1$ (see \cite[5.16]{manschot_bps}), or equivalently, for the divisor $H_{1,c}$ for $c\gg0$, which we denote by $H_{1,+\infty}$. Finally, we have to move from $H_{1,+\infty}$ to $H_{1,0}=C_0+f$ using the wall-crossing formula. This was done by Manschot \cite{manschot_betti,manschot_bps} in the case of rank $3$ sheaves.
\end{remark}
\providecommand{\bysame}{\leavevmode\hbox to3em{\hrulefill}\thinspace}
\providecommand{\href}[2]{#2}

 
\end{document}